\documentclass{amsart}

\usepackage[margin=1in]{geometry}

\usepackage{showlabels}
\usepackage{amsmath,amssymb}
\usepackage[hidelinks]{hyperref}

\newcommand{\Z}{\mathbf{Z}} 
\newcommand{\Q}{\mathbf{Q}}
 
\newcommand{\C}{\mathbf{C}} 
\newcommand{\F}{\mathbf{F}}

\DeclareMathOperator{\GCD}{GCD}

\newtheorem{theorem}{Theorem}
\newtheorem*{theorem*}{Theorem}

\newtheorem{lemma}[theorem]{Lemma}
\newtheorem{proposition}[theorem]{Proposition}
\newtheorem{corollary}[theorem]{Corollary}

\newtheorem{remark}[theorem]{Remark}

\begin{document}

\title{Zeta Functions of Lattices of the Symmetric Group}
\author{Tommy Hofmann}
\address{Department of Mathematics\\University of Kaiserslautern\\Postfach 3049\\67653 Kaiserslautern - Germany}
\email{thofmann@mathematik.uni-kl.de}


\begin{abstract}
  The symmetric group $\mathfrak S_{n+1}$ of degree $n+1$ admits an $n$-dimensional irreducible $\Q \mathfrak S_n$-module $V$ corresponding to the hook partition $(2,1^{n-1})$.
  By the work of Craig and Plesken we know that there are $\sigma(n+1)$ many isomorphism classes of $\Z \mathfrak S_{n+1}$-lattices which are rationally equivalent to $V$, where $\sigma$ denotes the divisor counting function.
  In the present paper we explicitly compute the Solomon zeta function of these lattices.
  As an application we obtain the Solomon zeta function of the $\Z \mathfrak S_{n+1}$-lattice defined by the Specht basis.
\end{abstract}

\maketitle 

\section{Introduction}

Let $\Lambda$ be a $\Z$-order in a finite-dimensional semisimple $\Q$-algebra and $M$ a $\Lambda$-lattice.
Following Louis Solomon (\cite{Solomon1977}) we define
\[ \zeta_\Lambda(M,s) = \sum_{ \substack{N \subseteq M \\ \text{$\Lambda$-sublattice}\\ \lvert M : N \rvert < \infty}} \lvert M : N \rvert^{-s}, \quad s \in \C, \]
to be the (Solomon) zeta function of $M$,
which is a natural generalization of the Riemann zeta function to the non-commutative setting.
In a series of papers Colin J. Bushnell and Irving Reiner have developed the theory of these zeta functions, including for example product formulas, functional equations or analytic continuation (see \cite{Bushnell1985} for an overview).
The case where $\Lambda = \Z G$ for some finite group $G$ (which was the motivation for Louis Solomon in the first place)  is of particular interest.
While the machinery of Colin J. Bushnell and Irving Reiner reveal the beauty and rich structure of the associated zeta functions, the actual computation for non-trivial groups $G$ and $\Z G$-lattices $M$ seems like a hard task.
In the past the case $M = \Lambda = \Z G$ has gained the most attention, resulting in the computation of $\zeta_{\Z G}(\Z G,s)$ for
\begin{itemize}
\item
$G$ cyclic of prime order (\cite{Solomon1977,Reiner1980,Bushnell1980}),
\item
$G$ cyclic of order $p^2$ for $p$ a prime (\cite{Reiner1980}),
\item
$G$ cyclic of order $p^3$ for $p$ a prime (\cite{Wittmann2004}),
\item
$G \in \{ C_2 \times C_2, C_3 \times C_3 \}$ (\cite{Takegahara1987}),
\item
$G$ dihedral of order $2p$ for $p$ an odd prime (\cite{Bushnell1981}),
\item
$G$ metacyclic (\cite{Hironaka1981,Hironaka1985}).
\end{itemize}
The aim of this paper is to extend the list of known zeta functions by an explicit computation for a family of irreducible $\Z G$-lattices, where $G$ is a symmetric group.
More precisely let $n \in \Z_{\geq 2}$ be a fixed integer and $G = \mathfrak S_{n+1}$ the symmetric group of degree $n+1$.
Corresponding to the hook partition $\lambda = (2,1^{n-1})$ of $n+1$ there exists the Specht module $V = S^\lambda$, an absolutely irreducible $\Q G$-module of dimension $n$.
Due to work of Maurice Craig (\cite{Craig1976}) we know that we can choose a $\Q$-basis $(e_1,\dotsc,e_{n})$ of $V$ such that for $1 \leq k \leq n$ the action of the adjacent transposition $(k~k+1)$ is given by multiplication with the matrix $E^{k,k-1} + 2 E^{k,k} + E^{k,k+1} - I_k$.
Here for $1 \leq i,j \leq n$ we denote by $E^{i,j}$ the matrix $(\delta_{ik}\delta_{jl})_{1\leq k,l,\leq n}$ and set $E^{i,j} = 0$ whenever $i$ or $j$ are not in $\{1,\dotsc,n\}$.
\par
With respect to this chosen representation, in \cite{Craig1976} representatives for the isomorphism classes of $\Z G$-lattices rationally equivalent to $V$ are constructed.
More precisely denote by $v$ the element $e_{n} + \sum_{i=1}^{n-1} (-1)^{n+1-i} i e_i \in V$.
For every integer $d \in \Z_{>0}$ define the $\Z$-module $L(d)$ via
\[ L(d) = \Bigl(\bigoplus_{i=1}^{n-1} \Z d e_i \Bigr) \oplus \Z v\]
that is, with respect to $(e_1,\dotsc,e_n)$ a basis matrix of $L(d)$ is 
\[ \begin{pmatrix} d & 0 & 0 & \dots & (-1)^{n} \cdot 1 \\
                   0 & d & 0 & \dots & (-1)^{n-1} \cdot 2 \\
                   \vdots & \vdots & \vdots & \vdots & \vdots \\
                   0 & \dots & 0 & d & n-1 \\
                   0 & \dots & 0 & 0 & 1 
                 \end{pmatrix} \in \operatorname{M}_{n}(\Z).
               \]
The main result of \cite{Craig1976} is the following classification.
Note that a similar result was obtained by Wilhelm Plesken in his PhD thesis (\cite{Plesken1974}).

\begin{theorem*}
  If $d$ is a divisor of $n+1$, then the $\Z$-module $L(d)$ is a $\Z G$-lattice of $V$.
  Moreover $\{ L(d) \, \mid \, d \text{ divides } n+1 \}$ is a complete set of representatives for the isomorphism classes of $\Z G$-lattices rationally equivalent to $V$.
\end{theorem*}

\par
The purpose of this paper is to construct for each of these lattices its Solomon zeta function, the main result being the following theorem.
As usual, for a prime $p$ we denote by $v_p \colon \Z \to \Z_{\geq 0}$ the $p$-adic valuation.

\begin{theorem}\label{thm:craig}
  For a divisor $d$ of $n+1$ the Solomon zeta function of the $\Z G$-lattice $L(d)$ is given by
  \[ \zeta_{\Z G}(L(d),s) = \zeta_\Q(ns) \prod_{p\mid n+1} \varphi_{p,d}(p^{-s}), \]
  where $p$ runs over all prime divisors of $n+1$, $\zeta_\Q$ denotes the classical Riemann zeta function and
  \[ \varphi_{p,d} = \sum_{j=0}^{v_p(d)} X^j + \sum_{j=v_p(d)+1}^{v_p(n+1)} X^{(j-v_p(d))(n-1)}.  \]
\end{theorem}

The Specht module $V$ has a distinguished basis $(e_T)_T$, the Specht basis, labeled by the standard Young tableaux of shape $\lambda$.
It is well known that with respect to this basis the representation matrices have only integral elements, turning the $\Z$-module
\[ L^{\lambda} = \bigoplus_{T} \Z e_T \]
into a $\Z G$-module, which we refer to as the Specht lattice corresponding to $\lambda$.
We will show that also the Solomon zeta function of this lattice---having a very simple form---is distinguished among all $\Z G$-lattices rationally equivalent to $V$.
More precisely, in Proposition~\ref{prop:findlattice} it is shown that $L^\lambda$ is isomorphic to $L(n+1)$, thus yielding the following corollary.

\begin{corollary}
  The Solomon zeta function of the Specht lattice $L^{(2,1^{n-1})}$ is given by
  \[ \zeta_{\Z G}(L^{(2,1^{n-1})}) = \zeta_\Q(ns) \prod_{p\mid n+1} \varphi_p(p^{-s}), \quad\text{ where}\quad \varphi_p = \frac{X^{v_p(n+1)} -1}{X - 1},  \]
and $p$ runs over all prime divisors of $n+1$.
\end{corollary}

To compute the Solomon zeta function we will rely on the technique introduced by Solomon in \cite{Solomon1977}.
Since this requires a complete understanding of the lattice of sublattices of $L(1)$, the second section is devoted to a detailed analysis of these objects, relying on the work of Craig.
We then use these information to compute the Solomon zeta functions by first computing the local Euler factors at the critical primes and then putting everything together.
Finally we locate the Specht lattice in the classification of Craig and in this way obtain its Solomon zeta function.

\section{The lattice of sublattices}

Let $p$ be a prime and $i \in \Z_{\geq 0}$ such that $p^i$ divides $n+1$.
The goal of this section is to determine the structure of the lattice of sublattices of $L(p^i)$ with index a $p$-power. 
We begin by determining all $\Z G$-sublattices of $L(1) = \bigoplus_{i=1}^{n} \Z e_i$ with index a $p$-power.

\begin{lemma}
  Let $L \subseteq L(1)$ be a $\Z G$-sublattice and assume there exists an exponent $c \in \Z_{\geq 1}$ such that $p^c L(1) \subseteq L$.
  Then there exist $a,b \in \Z_{\geq 0}$ with $a+b \leq c$, $p^b \mid n+1$ (i.e. $b \leq v_p(n+1)$) and $L = p^a L({p^b})$.
\end{lemma}

\begin{proof}
  Choose some basis matrix $B \in \mathrm{M}_{n}(\Z)$ of $L$ with respect to $(e_1,\dotsc,e_{n})$ and denote by $t$ the GCD of the entries of $B$.
  We write $B = t B'$ with $B' = (b'_{ij})_{ij}$ an element of $\mathrm{M}_{n}(\Z)$ satisfying $\GCD(b'_{ij} \mid i,j) = 1$.
  Denote by $L'$ the lattice spanned by the columns of $B'$.
  By assumption the index $\lvert L(1) : L \rvert = \det(B)$ is a $p$-power.
  Thus the same is true for $t$ and there exists $a \in \Z_{\geq 0}$ such that $t = p^a$ implying
  $p^{c-a} L(1) \subseteq L'$.
  Now \cite[Lemma 9]{Craig1976}, applied to $L'$ with basis matrix $B'$, shows that there exists an integer $b \in \Z_{\geq 0}$ with $b \leq c-a$, $p^b \mid n+1$ and $L' = L({p^b})$.
\end{proof}

\begin{lemma}\label{lem:2}
  Let $a,b,a',b' \in \Z_{\geq 0}$ be integers.
  Then the following hold:
  \begin{enumerate}
  \item
    We have $p^a L(p^b) \cap p^{a'} L(p^{b'}) = p^{\max(a,a')} L(p^{\max(a+b,a'+b')-\max(a,a')})$.
  \item
    The inclusion $p^a L({p^b}) \subseteq p^{a'} L({p^{b'}})$ holds if and only if $a \geq a'$ and $a+b \geq a'+b'$.
  In this case we have $\lvert p^{a'}L({p^{b'}}) : p^{a}L({p^b}) \rvert = p^{(a-a')n + (b-b')(n-1)}$.
  \end{enumerate}
\end{lemma}

\begin{proof}
  We have
  \[ p^a L({p^b}) = \Bigl(\bigoplus_{i=1}^{n-1} \Z p^{a+b}e_i\Bigr) \oplus \Z p^a v \text{ and } p^{a'} L({p^{b'}}) = \Bigl( \bigoplus_{i=1}^{n-1} \Z p^{a'+b'}e_i \Bigr) \oplus \Z p^{a'} v . \]
  By comparing the coefficients in front of $v$ and the $e_i$, the claims follow.
\end{proof}

\begin{lemma}\label{lem:neu}
  Let $0 \leq i \leq v_p(n+1)$ and assume that $L$ is a proper $\Z G$-sublattice of $L(p^i)$ with index a $p$-power.
  Then following hold:
  \begin{enumerate}
  \item
    If $i = 0$, then $L$ is contained in $L(p)$.
  \item
    If $0 < i < v_p(n+1)$, then $L$ is contained in $L(p^{i+1})$ or $pL(p^{i-1})$.
  \item
    If $i = v_p(n+1)$, then $L$ is contained in $pL(p^{v_p(n+1)-1})$.
  \end{enumerate}
\end{lemma}

\begin{proof} Write $L = p^a L(p^b)$ with $a,b \in \Z_{\geq 0}$ and $b \leq v_p(n+1)$.\\
  (1): Assume that $L$ is a proper $\Z G$-sublattice of $L(1)$. Then $(a,b) \neq (0,0)$ and therefore $a+b \geq 1$. 
  By Lemma~\ref{lem:2} we have $L \subseteq L(p)$.\\
  (2): Assume that $L$ is a proper $\Z G$-sublattice of $L(p^i)$. Then $(a,b) \neq (0,i)$.
  If $a \neq 0$, then $a \geq 1$ and $a + b \geq i = 1 + (i-1)$. Thus by Lemma~\ref{lem:2} we have $L \subseteq p L(p^{i-1})$.
  If $a = 0$, then we necessarily have $b > i$ and therefore $L = L(p^b) \subseteq L(p^{i+1})$. \\
  (3): 
  Assume that $L$ is a proper $\Z G$-sublattice of $L(p^{v_p(n+1)})$.
  Then $(a,b) \neq (0,v_p(n+1))$.
  Since we also have $b \leq v_p(n+1)$, we conclude that $a \geq 1$. 
  Thus by Lemma~\ref{lem:2} we obtain $L \subseteq p L(p^{v_p(n+1)-1})$.
\end{proof}

Let $N$ be a maximal $\Z G$-sublattice of a $\Z G$-lattice $M$.
By invoking the Krull--Azumaya theorem (also known as Lemma of Nakayama) it follows that there exists a unique rational prime $q$ such that $q M \subseteq N$.
In this case we call $N$ a \textit{$q$-maximal} sublattice of $M$.
The finite set of all $q$-maximal $\Z G$-sublattices of $M$ will be denoted by $\max_q(M)$.
Using this notion, Lemma~\ref{lem:neu} takes the following form.

\begin{lemma}\label{lem:7}
  For $0 \leq i \leq v_p(n+1)$ we have
  \[ \max\nolimits_p(L({p^i})) = \begin{cases} \{ L(p) \}, \quad & \text{if $i=0$},\\
      \{L({p^{i+1}}),pL({p^{i-1}})\}, \quad & \text{if $0 < i < v_p(n+1)$},\\
      \{ pL({p^{i-1}}) \}, \quad & \text{if $i = v_p(n+1)$}.
    \end{cases}
  \]
\end{lemma}

For a $\Z G$-lattice $M$ of $V$ we now define
\[ \operatorname{rad}_p(M) = \bigcap_{N \in \max_p(M)} N \]
to be the \textit{$p$-radical of $M$}.
By extension of scalars we associate to $M$ the $\Z_p G$-module $M \otimes_{\Z G} \Z_p$, which we denote by $M_p$.
We also define
\begin{align*} \Phi_p(M) &= \{ L \text{ a $\Z G$-sublattice of $M$} \mid \operatorname{rad}_p(M) \subseteq L \subseteq M \} \text{ and }\\
  \Phi_p(M,N)  &= \{ L \in \Phi_p(M) \mid L_p \cong N_p \text{ as $\Z_p G$-lattices } \}  \end{align*}
for any $\Z G$-lattice $N$ of $V$.

\begin{remark}\label{remark:1}
Let $M$ and $N$ be sublattices of $L(1)$ with index a $p$-power and $L \in \Phi_p(M,N)$.
Then the condition $L_p \cong N_p$ as $\Z_p G$-lattices is equivalent to $L \cong N$ as $\Z G$-lattices.
This can be seen as follows. Assume that $L_p \cong N_p$ as $\Z_p G$-modules.
If $q \neq p$ is a rational prime, we have---the indices $|L(1) : M|$, $|L(1) : N|$ and $|L(1) : L|$ being $p$-powers---$L(1)_q = M_q = N_q = L_q$.
Thus $L$ and $N$ lie in the same genus.
As $V$ is absolutely irreducible this implies $L \cong N$ as $\Z G$-lattices (see \cite[31.26 Theorem]{Curtis1981}).
Thus
\[ \Phi_p(M,N)  = \{ L \in \Phi_p(M) \mid N \cong L \text{ as $\Z G$-lattices } \}. \]
\end{remark}

\begin{lemma}\label{lem:4}
  For $0 \leq i \leq v_p(n+1)$ we have
  \[ \operatorname{rad}_p(L) = \begin{cases} L(p), \quad & \text{if $i=0$}, \\
      pL({p^i}), \quad & \text{if $0 < i < v_p(n+1)$}, \\
      pL({p^{i-1}}), \quad & \text{if $i = v_p(n+1)$}.
    \end{cases} \]
\end{lemma}

\begin{proof}
  This follows from Lemma~\ref{lem:2} and~\ref{lem:7}.
\end{proof}

\begin{lemma}
  For $0 \leq i \leq v_p(n+1)$ the following hold:
  \[\Phi_p(L(p^i)) = \begin{cases} \{  L(1),L(p) \}, \quad & \text{if $i =0$}, \\
        \{ pL(p^{i-1}), L(p^i), pL(p^i), L(p^{i+1}) \}, \quad &\text{if $0 < i < v_p(n+1)$},\\
           \{ pL(p^{i-1}), L(p^{i}) \}, \quad &\text{if $i=v_p(n+1)$} \end{cases} \]
\end{lemma}

\begin{proof}
  By Lemma~\ref{lem:4} we know the $p$-radical of $L(p^i)$.
  The result follows by applying Lemma~\ref{lem:2}.
\end{proof}

We can now determine $\Phi_p(L(p^i),L(p^j))$ for all $0 \leq i$, $j \leq v_p(n+1)$.

\begin{lemma}\label{lem:3}
  The following hold:
  \begin{enumerate}
    \item
We have
      \[ \Phi_p(L(1),L({p^j})) = \begin{cases} \{ L(1) \}, \quad & \text{if $j=0$},\\
          \{ L(p) \} , \quad & \text{if $j=1$},\\
          \emptyset , \quad &\text{otherwise}.
        \end{cases} \]
    \item
      For $1 < i < v_p(n+1)$ we have
      \[ \Phi_p(L({p^i}),L({p^j})) = \begin{cases}
          \{ pL({p^{i-1}}) \}, \quad &\text{if $j = i-1$},\\
          \{ L({p^i}),pL({p^i}) \}, \quad & \text{if $j = i$},\\
          \{ L({p^{i+1}}) \}, \quad & \text{if $j=i+1$},\\
          \emptyset, \quad & \text{otherwise}.
        \end{cases}
      \]
    \item
We have
      \[ \Phi_p(L({p^{v_p(n+1)}}),L({p^j})) =
        \begin{cases}
          \{ pL({p^{v_p(n+1)-1}})\}, \quad & \text{if $j=v_p(n+1)-1$},\\
          \{L({p^{v_p(n+1)}})\}, \quad & \text{if $j=v_p(n+1)$},\\
          \emptyset, \quad & \text{otherwise}.
        \end{cases}\]
    \end{enumerate}
\end{lemma}

\begin{proof}
  Note that by the classification of Craig we know that for $0 \leq i,j \leq v_p(n+1)$ the $\Z G$-lattices $L(p^i)$ and $L(p^j)$ are isomorphic if and only if $i = j$.
  Since for any $\Z G$-lattice $M$ and integer $m \in \Z \setminus \{ 0 \}$ the $\Z G$-lattices $M$ and $mM$ are isomorphic, it follows, by Remark~\ref{remark:1}, that $(p^a L(p^i))_p \cong (p^b L(p^j))_p$ as $\Z_p G$-lattices if and only if $i = j$.
\end{proof}

\section{Computation of the local factors}

We now turn to the computation of the Solomon zeta functions which naturally splits into two parts: The computation of local Euler factors and the global step, where we piece together the local information.
We begin with the local computation as described in \cite{Solomon1977}.
For a $\Z_p G$-lattice $M$, where $G$ is a finite group, we define the formal Dirichlet series
\[ \zeta_{\Z_p G}(M,s) = \sum_{\substack{L \subseteq M\\\text{$\Z_p G$-sublattice}\\ \lvert M : L \rvert < \infty}} \lvert M : L \rvert^{-s}, \quad s \in \C, \]
similar to the Solomon zeta function of $\Z G$-lattices.
Since in this local situation all sublattice indices occurring are powers of $p$, we can reformulate the zeta function in terms of power series.
More precisely for a $\Z_p G$-sublattice $L$ of $M$ with $\lvert M : L \rvert = p^i$ we set $[ M : L ] = X^i \in \Z[X]$ and we define
\[ Z(M,X) = \sum_{ \substack{L \subseteq M \\ \text{$\Z_p G$-sublattice}}} [ M : L] \in \Z[[X]].\]
The connection to the original zeta function is given by the relation 
\[ Z(M,p^{-s}) = \zeta_{\Z_p G}(M,s). \]
Let $M = M_0,\dotsc,M_r$ be a set of representatives for the isomorphism classes of $\Z_p G$-lattices in $\Q_p \otimes_{\Z_p} M$.
For $i=0,\dotsc,r$ we define the partial zeta function
\[ Z(M,M_i,X) = \sum_{\substack{L \subseteq M \\ \text{$\Z_p G$-sublattice} \\ L \cong M_i }} [ M : L ] \in \Z[[X]] \]
and obtain
\[ Z(M,X) = \sum_{i=0}^r Z(M,M_i,X). \]
Thus as soon as we know the partial functions $Z(M_i,M_j,X)$ we know $Z(M,X)$.
Instead of trying to compute the zeta function of $M$ alone, Solomon had the idea to compute $Z(M_i,M_j,X)$ for all $0 \leq i,j \leq r$ simultaneously.
To be more precise, let us denote by $\mathbf B$ the matrix $(Z(M_i,M_j,X))_{0 \leq i ,j \leq r} \in \operatorname{M}_{r+1}(\Z[[X]])$.
To understand the structure of $\mathbf B$, Solomon gave a fairly explicit procedure to construct the inverse of $\mathbf B$.
\par
To describe this construction, we need a little more notation.
Let $\max(M) = \{ N_1,\dotsc,N_k\} $ be the maximal $\Z_p G$-sublattices of $M$.
Set $\operatorname{rad}(M) = \bigcap_{N \in \max(M)} N$ and $\Phi(M) = \{ N \mid \operatorname{rad}(M) \subseteq N \subseteq M \}$.
For a $\Z_p G$-lattice $N$ set $\Phi(M,N) = \{ L \in \Phi(M) \mid L \cong N\}$.
Finally for $L \in \Phi(M)$ define
\[ \mu(M,L) = \sum_{ J} (-1)^{\lvert J \rvert}, \]
where the sum runs over all subsets $J$ of $\{1,\dotsc,k\}$ with $\bigcap_{j \in J} N_j = L$.

\begin{lemma}[{\cite[Lemma 3]{{Solomon1977}}}]
  The matrix $\mathbf A = (A_{ij})_{0\leq i,j\leq r} \in \operatorname{M}_{r+1}(\Z[X])$ defined by
  \[ A_{ij} = \sum_{L \in \Phi(M_i,M_j)} \mu(M_i,L) [ M_i : L] \in \Z[X] \]
  is the inverse of $\mathbf B$.
\end{lemma}

We now want to apply this theorem to the situation in which the $\Z_p G$-lattices involved are all of the form $M_p$, for some $\Z G$-lattice $M$.
By definition, the determination of $\mathbf A$ requires knowledge about the sublattices of $M_p$.
By exploiting the connection between $\Z_p G$-sublattices of $M_p$ and the $\Z G$-sublattices of $M$ with index a $p$-power, the following well-known property of $p$-adic completion and its corollary show that in fact any computation with $\Z_p G$-lattices can be avoided.

\begin{lemma}\label{lem:9}
  Let $M$ be a $\Z G$-lattice. 
  Then the map
  \begin{align*} \Psi \colon \{ N \subseteq M \text{a $\Z G$-sublattice with index a $p$-power}\} & \longrightarrow \{ N \subseteq M_p \text{ a $\Z_p G$-sublattice} \} \\
    N & \longmapsto N_p = N \otimes_{\Z G} \Z_p
    \end{align*} 
    is an isomorphism of lattices with $\lvert M : N \rvert = \lvert M_p : N_p \rvert$.
\end{lemma}

If $N$ is a $\Z G$-sublattice of $M$ with index a $p$-power, say $\lvert M  : N \rvert = p^i$, we define $[ M : N ] = X^i \in \Z[X]$.
For $N \in \Phi_p(M)$ and $\max_p(M) = \{ N_1,\dotsc,N_r \}$ we set
\[ \mu_p(M,N) = \sum_{J} (-1)^{\lvert J \rvert}, \]
where $J$ runs through all subsets of $\{ 1,\dotsc,r\}$ with $\bigcap_{j \in J} N_j = N$.

\begin{corollary}\label{coro:10}
  Let $M$ and $N$ be $\Z G$-lattices.
  Then the following hold:
  \begin{enumerate}
  \item
    The maps
    \begin{align*} \Phi_p(M) &\longrightarrow \Phi(M_p),\, L \longmapsto L_p ,\\
       \Phi_p(M,N) &\longrightarrow \Phi(M_p,N_p), \, L \longmapsto L_p \end{align*}
    are bijections.
  \item
    If $L \in \Phi_p(M)$ then $\mu_p(M,N) = \mu(M_p,N_p)$.
  \end{enumerate}
\end{corollary}

Let us come back to the situation of Section~1 with $G$ being $\mathfrak S_{n+1}$ and the $\Z G$-lattices being contained in the Specht module $V$.
We want to apply the preceding discussion to the computation of $Z(L(p^i)_p,X)$.
First of all note that $\{ L(p^i)_p \mid 0 \leq i \leq v_p(n+1)\} $ is a complete set of representatives for the isomorphism classes of $\Z_p G$-lattices of $\Q_p \otimes_{\Z_p}  M$.
This can be seen as in Remark~\ref{remark:1}.
Thus it is sufficient to compute the matrix $\mathbf A = (A_{ij})_{0 \leq i,j \leq v_p(n+1)}$, where
\[ A_{ij} = \sum_{ L \in \Phi(L(p^i)_p,L(p^j)_p)} \mu(L(p^i)_p,L) [ L(p^i)_p : L) ], \]
for all $0 \leq i,j \leq v_p(n+1)$.
  The following following is a consequence of Lemma~\ref{lem:9} and Corollary~\ref{coro:10}.

\begin{lemma}
  For $0 \leq i,j \leq v_p(n+1)$ we have
  \[ A_{ij} = \sum_{ L \in \Phi_p(L(p^i),L(p^j)) } \mu_p(L(p^i),L) [ L(p^i) : L ]. \]
\end{lemma}

\begin{proposition}
  The matrix $\mathbf A = (A_{ij})_{0 \leq i,j \leq v_p(n+1)}$ is a tridiagonal matrix with $A_{ij} = 0$ if $\lvert i - j \rvert > 1$ and
  \[ A_{ij} = \begin{cases} 1, \quad & \text{if $i= j = 0$ or $i=j = v_p(n+1)$},\\
      1+X^{n}, \quad & \text{if $0 < i = j < v_p(n+1)$},\\
      -X, \quad & \text{if $j = i - 1$},\\
      -X^{n-1}, \quad & \text{if $j = i + 1$},
    \end{cases}
  \]
  that is,
  \[ \mathbf A = \begin{pmatrix} 1 & -X^{n-1} & 0 & 0 & 0 & \dots & 0 \\
                        -X & 1+X^{n} & -X^{n-1} & 0 & 0 & \dotsb & 0 \\
                        0 & -X & 1+X^{n} & -X^{n-1} & 0 & \dots & 0 \\
                        \vdots & \vdots & \vdots & \vdots & \vdots & \vdots & \vdots \\
                        0 & \dots & 0 & -X & 1 + X^{n} & -X^{n-1} & 0   \\
                        0 & \dots & 0 & 0 & -X & 1 + X^{n} & -X^{n-1} &   \\
                        0 & \dots & 0 & 0 & 0 & -X & 1 
                      \end{pmatrix}. \]
                    
\end{proposition}

\begin{proof}
  Apply Corollary~\ref{coro:10} together with Lemma~\ref{lem:3} of the previous section.
\end{proof}

\begin{proposition}
  The matrix $\mathbf B = (B_{ij})_{0 \leq i,j \leq v_p(n+1)} \in \operatorname{M}_{(v_p(n+1)+1)}(\Z[X])$ with $B_{ij} = Z(L(p^i)_p,L(p^j)_p,X)$ is given by
  \[ B_{ij} = \frac 1 {1 - X^{n}} \begin{cases}
      1, \quad & \text{if $i=j$},\\
      X^{(i-j)(n-1)}, \quad & \text{if $i < j$},\\
      X^{j-i}, \quad &\text{if $i > j$}, 
    \end{cases}
  \]
 that is, 
 \[ \mathbf B = \frac{1}{1-X^{n}}\begin{pmatrix} 1 & X^{n-1}  & X^{2(n-1)} & X^{3(n-1)} & X^{4(n-1)} & \dotsb & X^{(v_p(n+1))(n-1)} \\
     X & 1 & X^{n-1} & X^{2(n-1)} & X^{3(n-1)} & \dotsb & X^{(v_p(n+1)-1)(n-1)} \\
     X^2 & X & 1 & X^{n-1} & X^{2(n-1)} & \dotsb & X^{(v_p(n+1)-2)(n-1)} \\
     \vdots &  \vdots &\vdots &\vdots &\vdots & \vdots & \vdots \\
     X^{v_p(n+1)-2} & X^{v_p(n+1)-3} & \dotsb & X & 1 & X^{n-1} & X^{2(n-1)} \\
     X^{v_p(n+1)-1} & X^{v_p(n+1)-2} & X^{v_p(n+1)-3} & \dotsb & X & 1 & X^{n-1} \\
     X^{v_p(n+1)} & X^{v_p(n+1)-1} & X^{v_p(n+1)-2} & \dotsb & X^2 & X & 1 
   \end{pmatrix}.\]
\end{proposition}

\begin{proof}
  This is a straight forward calculation.
  For $0 \leq i \leq v_p(n+1)$ denote by $A_i$ the $i$th row of $A$ and by $B_i$ the $i$th column of $B$.
  We need to show that $A_i B_j= \delta_{ij}$ for $0 \leq i,j\leq v_p(n+1)$.
  As an example let us verify the case $1 \leq i,j \leq v_p(n+1)-1$, for which we have $A_i B_j = A_{i,i-1}B_{j,i-1} + A_{i,i} B_{j,i} + A_{i,i+1}B_{j,i+1} = -X B_{j,i-1} + (1+X^n)B_{j,i} - X^{n-1}B_{j,i+1}$.
  Now
  \[ (B_{j,i-1},B_{j,i},B_{j,i+1}) = \begin{cases} (X^{(j-i+1)(n-1)},X^{(j-i)(n-1)},X^{(j-i-1)(n-1)}), \quad &\text{if $i<j$}\\
                                                   (X^{n-1},1,X), \quad &\text{if $i=j$}\\
                                                   (X^{i-j-1},X^{i-j},X^{i-j+1}), \quad&\text{if $i>j$}.
                                                 \end{cases}
                                               \]
  and we easily conclude that $A_i B_j = \delta_{ij}$.
\end{proof}

\begin{corollary}\label{coro:2}
  For $0 \leq i \leq v_p(n+1)$ we have $\zeta_{\Z_p G}(L(p^i)_p,s) = \varphi_{p,i}(p^{-s})$, where
  \[ \varphi_{p,i} = \frac 1 {1-X^{n}} \Biggl( \sum_{j=0}^i X^{j} + \sum_{j=i+1}^{v_p(n+1)} X^{(j-i)(n-1)} \Biggr). \]
\end{corollary}

\begin{proof}
  The claim follows from $Z(L(p^i)_p,X) = \sum_{j=0}^{v_p(n+1)} B_{ij}$ and $\zeta_{\Z_p G}(L(p^i)_p,s) = Z(L(p^i)_p,p^{-s})$.
\end{proof}

\section{Computation of the Solomon zeta functions}

We now return to the global Solomon zeta function.
Again, by $G$ we denote $\mathfrak S_{n+1}$ and by $V$ the irreducible $n$-dimensional Specht module of the previous sections.
For all $\Z G$-lattices rationally equivalent to $V$ we want to compute 
\[ \zeta_{\Z G}(M,s) = \sum_{ \substack{\{ 0 \} \neq N \subseteq M \\ \text{$\Z G$-sublattice}}} \lvert M : N \rvert^{-s}, \quad s \in \C. \]
The connection to the local zeta functions is given by the following formula,
\[ \zeta_{\Z G}(M,s) =  \prod_{p} \zeta_{\Z_p G}(M_p,s), \]
a generalization of the Euler product of the Riemann zeta function.
The factor corresponding to $p$ on the right hand side is called the local factor or Euler factor at $p$.
\par
By the work of Solomon we know that for almost all primes the local factor of $\zeta_{\Z G}(M,s)$ depends only on $V$ and not on the chosen $\Z G$-lattice.
To see this we first introduce
\[ \zeta_V(s) = \zeta_\Q(ns) \]
and for every rational prime $p$ the local factor 
\[ \zeta_{V,p}(s) = \frac{1}{1-p^{-ns}}. \]
Thus we have
\[ \zeta_V(s) = \prod_{p} \zeta_{V,p}(s), \]
where $p$ runs through all rational primes.
Now if $M$ is a $\Z G$-lattice of $V$, by \cite[Theorem 1]{Solomon1977} we know that
\[ \zeta_{\Z G}(M,s) = \zeta_V(s) \prod_{p \mid \lvert G \rvert} \frac{\zeta_{\Z_p G}(M_p,s)}{\zeta_{V,p}(s)}. \]
As $\lvert G \rvert = (n+1)!$ we have 
\[ \zeta_{\Z G}(M,s) = \zeta_V(s) \prod_{\substack{p\text{ prime} \\ p \leq n+1}} \frac{\zeta_{\Z_p G}(M_p,s)}{\zeta_{V,p}(s)} \]
and it remains to determine the local factors for all rational primes $p \leq n+1$.

\begin{proposition}\label{prop:trivial}
  For all rational primes $p$ not dividing $n+1$ we have
  \[ \frac{\zeta_{\Z_p G}(M_p,s)}{\zeta_{V,p}(s)} = 1. \]
\end{proposition}

\begin{proof}
  If $p$ does not divide $n+1$, then by \cite[Theorem 23.7]{James1978} for all lattices $M$ of $V$ the quotient $M/pM$ is an irreducible $\F_p G$-module.
  In particular $M/pM \cong M_p/p M_p$ is irreducible.
  This implies that $p M_p$ is the only maximal $\Z_p G$-sublattice of $M_p$ and therefore $\{ p^i M_p \, \mid \, i \in \Z_{\geq 0} \}$ is the set of all $\Z_p G$-sublattices of $M_p$.
  Hence
  \[ \zeta_{\Z_p G}(M_p,s) = \sum_{i=0}^\infty (p^{in})^{-s} = \frac{1}{1-p^{-ns}} = \zeta_{V,p}(s).\qedhere\]
\end{proof}

We can now prove the main result.

\begin{proof}[Proof of Theorem~\ref{thm:craig}]
  Let $m = (n+1)!/n$.
  The proof of \cite[Lemma 10]{Craig1976} and its corollary show that
  \[ L(d) = \sum_{p \mid n+1} m_p L(p^{v_p(d)}), \]
  where $m_p = m/p^{v_{p}(m)}$.
  While $(m_p L(p^{v_p(d)}))_p = L(p^{v_p(d)})_p$, for a prime $q \neq p$ we have $(m_q L(q^{v_q(d)}))_{p} = p^{v_{p}(m)} L(1)_{p}$.
  Since $p^{v_p(m)} L(1) \subseteq p^{v_p(d)} L(1) \subseteq L(p^{v_p(d)})$ we conclude that
  \[ L(d)_{p} = L(p^{v_p(d)})_{p}. \]
  Together with Proposition~\ref{prop:trivial} we obtain
  \[ \zeta_{\Z G}(L(d),s) = \zeta_\Q(ns) \prod_{p \mid n +1} \frac{\zeta_{\Z_p G}(L(d)_p,s)}{\zeta_{V,p}(s)} = \zeta_\Q(ns) \prod_{p \mid n +1} \frac{\zeta_{\Z_p G}(L(p^{v_p(d)})_p,s)}{\zeta_{V,p}(s)}. \]
  Since $\zeta_{V,p}(s)^{-1} = g(p^{-s})$, where $g = 1 - X^n$, the claim follows from Corollary~\ref{coro:2}.
\end{proof}

\section{The natural lattice of the Specht module}

The Specht module $V$ corresponding to the partition $\lambda = (2,1^{n-1})$ of $n+1$ has a distinguished basis $(e_T)_T$, the Specht basis, labeled by the standard Young tableaux of shape $\lambda$.
With respect to this basis the representation matrices are integral, turning the $\Z$-module
\[ L^\lambda = \bigoplus_{T} \Z e_T \]
into a $\Z G$-lattice of $V$, which we call the Specht lattice corresponding to $\lambda$.
To compute the Solomon zeta function of this Specht lattice, it is sufficient to locate it in the classification of Craig.

\begin{lemma}\label{lem:specht}
  If the rational prime $p$ is a divisor of $n+1$, the module $L^\lambda$ has a unique $p$-maximal sublattice. This sublattice has index $p$ in $L^\lambda$.
\end{lemma}

\begin{proof}
  It is sufficient to prove that the $\F_p G$-module $L^\lambda/pL^\lambda$ has a unique maximal submodule and that this submodule has index $p$.
  Let $S^\lambda$ be the Specht module over $\F_p$ corresponding to $\lambda$.
  Then $L^\lambda/pL^\lambda$ is isomorphic to $S^\lambda$ and we can use the rich theory of Specht modules.
  Let $\mu$ be the conjugate partition of $\lambda$, that is, $\mu = (n,1)$.
  Then by \cite[Theorem 8.15]{James1978} we have 
  \[ (S^\lambda)^\ast \cong S^\mu \otimes S^{(1^{n+1})}, \]
  where $(S^\lambda)^\ast = \operatorname{Hom}_{\F_p}(S^\lambda,\F_p)$ is the dual of $S^\lambda$.
  As $n \geq 2$ the partition $\mu$ is $p$-regular and by \cite[Corollary 12.2]{James1978} the module $S^{\mu}$ has the unique composition series
  \[ S^{\mu} \supseteq  S \supseteq \{0\}. \]
  where $S$ is isomorphic to the trivial $\F_p G$-module.
  Thus $S^{\mu}$ has a unique $1$-dimensional minimal submodule.
  As $S^{(1^{n+1})}$ is $1$-dimensional, also the tensor product $S^{\mu}\otimes S^{(1^{n+1})}$ has a unique $1$-dimensional minimal submodule.
  Since 
  \[ N \longmapsto N^\perp = \{ f \in (S^\lambda)^\ast \, | \, f(n) = 0 \text{ for all $n \in N$}\} \]
  is an anti-isomorphism between the lattice of submodules of $S^\lambda$ and the lattice of submodules of $(S^\lambda)^\ast$ with $\dim_{\F_p}(N^\perp) = \dim_{\F_p}(S^\lambda) - \dim_{\F_p}(N)$ (see \cite[Proposition 4.1.1]{Hazewinkel2007}), the module $S^\lambda$ has a unique maximal submodule of dimension $n-1$.
\end{proof}

\begin{proposition}\label{prop:findlattice}
  We have $L^\lambda \cong L(n+1)$ as $\Z G$-modules.
\end{proposition}

\begin{proof}
  Assume that they are not isomorphic.
  Then there exists a prime divisor $p$ of $n+1$ such that $(L^\lambda)_p$ and $(L(n+1))_p$ are not isomorphic as $\Z_p G$-modules.
  Since $(L(n+1))_p = (L^{v_p(n+1)})_p$ there exists $0 \leq i < v_p(n+1)$ such that $(L^\lambda)_p \cong L(p^i)_p$. 
  By Lemma~\ref{lem:7} this implies that $L^\lambda$ has a $p$-maximal submodule of index $>p$, contradicting Lemma~\ref{lem:specht}.
\end{proof}

This proves Corollary~2.

\section*{Acknowledgement}

The author is grateful to Claus Fieker and Susanne Danz for many valuable discussions, and to the anonymous referee for helpful comments.

\bibliographystyle{amsalpha}

\bibliography{paper_zeta}

\end{document}